\documentclass[final,1p]{elsarticle}
\journal{arXiv}
\usepackage{graphicx}
\usepackage{epstopdf}
\usepackage{amsmath}
\usepackage{amssymb}
\usepackage{enumerate}
\usepackage{comment}
\usepackage{geometry}

\newtheorem{theorem}{Theorem}
\newtheorem{lemma}[theorem]{Lemma}
\newtheorem{corollary}[theorem]{corollary}
\newtheorem{definition}{Definition}
\newproof{proof}{Proof}

\newcommand{\req}[1]{Eq.\,(\ref{#1})}

\begin{document}

\begin{frontmatter}

\title{A Posteriori Error Bounds for Two Point Boundary Value Problems: A Green's Function Approach}
\author{Jeremiah Birrell\footnote{email: jbirrell@email.arizona.edu}}
\address{Program in Applied Mathematics, The University of Arizona, Tucson,  Arizona, 85721, USA}

\begin{abstract}
We present a computer assisted method for generating existence proofs and a posteriori error bounds for solutions to two point boundary value problems (BVPs).  All truncation errors are accounted for and, if combined with interval arithmetic to bound the rounding errors, the computer generated results are mathematically rigorous.  The method is formulated for $n$-dimensional systems and does not require any special form for the vector field of the differential equation.  It utilizes a numerically generated approximation to the BVP fundamental solution and Green's function  and thus can be applied to stable BVPs whose initial value problem is unstable. The utility of the method is demonstrated on a pair of singularly perturbed model  BVPs and by using it to rigorously show the existence of a periodic orbit in the Lorenz system.  
\end{abstract}

\begin{keyword}
two point boundary value problems,  computer assisted proofs, periodic orbits, a posteriori error analysis, singular perturbations\\
\hspace{-2mm} \textit{ 2010 Mathematics Subject Classification.} Primary: 37C27, 65G20, 34B15; Secondary: 34B27, 65L10, 65L11.
\end{keyword}

\end{frontmatter}

\section{Introduction}
We propose a new computer assisted method for rigorously proving invertibility and bounding the norm of the inverse of operators of the form\footnote{Conditions on $A(t)$ and the precise functions spaces in which we work will be specified later.} 
\begin{equation}\label{lin_bvp}
F[v](t)=\left(v(t)-v(0)-\int_0^t A(s)v(s)ds,B_0v(0)+B_1v(1)\right),\hspace{2mm} A(t),B_i\in\mathbb{R}^{n\times n}.
\end{equation}
Operators of this form correspond to linear two point boundary value problems (BVPs) and arise naturally as the Fr\'{e}chet derivative of operators defining nonlinear BVPs.  Bounds on the norm of the inverse of operators of this type are required in a posteriori proofs of existence and local uniqueness of nonlinear equations via the Newton-Kantorovich theorem~\cite{kantorovich1964functional}.

Computer assisted methods have been successfully applied to many problems in differential equations and dynamical systems, including chaos in the R\"{o}ssler \cite{0951-7715-10-1-016} and Lorenz  \cite{Galias1998165} systems, the existence of the Lorenz  attractor \cite{Tucker19991197}, flow (in)stability \cite{Asen11981,watanabe2009computer}, existence of heteroclinic~\cite{Lessard2014}, homoclinic~\cite{doi:10.1137/100812008,doi:10.1137/12087654X,Arioli201551}, and periodic~\cite{Coomes,radii_poly} orbits, solution of the Feigenbaum equation~\cite{koch}, and the (in)stability of matter~\cite{fefferman1986relativistic}.  Generalizations to elliptic and parabolic PDEs are discussed in, for example, \cite{Nakao1991323,nagatou1999approach,nakao2001numerical,plum2008existence}. As it pertains to two point boundary value problems, computer assisted proofs are  typically based on a fixed point theorem \cite{Yamamoto,Plum,Nakao1992489}.  In particular, various forms of the Newton-Kantorovich theorem were used in \cite{McCarthy,Kedem,Takayasu}.  A method based on differential inequalities as presented in \cite{Gohlen}, a method using radii polynomials is given in~\cite{radii_poly}, and \cite{doi:10.1137/13090883X} utilized Chebyshev  series.

The methods in \cite{McCarthy,Gohlen,Plum,Takayasu,Nakao1992489,Fogelklou20111227,Fogelklou2012} focused on second order  equations; \cite{McCarthy} worked with an equivalent integral formulation while \cite{Takayasu,Nakao1992489} used a finite element based method. \cite{Fogelklou2012} treated an integral boundary condition. The results in \cite{Urabe,Kedem}  utilized the theory of initial value problem (IVP) fundamental solutions to treat $n$-dimensional two point BVPs, the former focusing on the periodic case. 

 Our method most closely resembles that of Ref.~\cite{Kedem} and applies to $n$-dimensional systems without any special assumptions on the form of the vector field.  The present method differs from \cite{Kedem} in the use of the BVP fundamental solution and Green's function as the primary tool, as opposed to the IVP fundamental solution.  See also \cite{doi:10.1137/140987973}, which used a Green's function method to solve a singular problem on an unbounded domain.

A general outline of our method is as follows.
\begin{enumerate}
\item Generate (non-rigorous) approximations, $\tilde \Phi_j$, $\tilde G_{i,j}$ to the BVP fundamental solution and Green's function, respectively, on a mesh.
\item Extend these to piecewise polynomial approximations $\tilde\Phi(t)$ and $\tilde G(t,s)$ on the whole domain.
\item Use $\tilde \Phi$ and $\tilde G$ to define an approximate inverse, $H$, to the BVP operator $F$.
\item Using estimates involving $H$, prove existence of an exact inverse to $F$ and derive a machine computable formula that bounds the norm of $F^{-1}$.
\end{enumerate}
 The use of the BVP fundamental solution makes the method presented here applicable in a range of situations for which IVP based methods are ill suited, namely when the  IVP fundamental solution or its inverse has a `large' norm but the corresponding BVP fundamental solution and Green's function has a `moderately sized' norm.  This is often the case in singularly perturbed problems and so the  effectiveness of the method is demonstrated on two such examples in Section \ref{sec:num_tests}. We will refer to these imprecisely defined `large' and `moderate' regimes as unstable and stable respectively, similar to the discussion in \cite{ascher_petzold}.

To obtain completely rigorous results, some means of bounding the floating point rounding error is needed, typically based on the theory of interval arithmetic.  There are many references from which to learn more about validated numerics and  interval arithmetic, for example~\cite{moore1979methods,kearfott1996interval,jaulin2001applied,moore2009introduction,tucker2011validated}.  Results in this paper are presented so that if interval arithmetic is used for all arithmetic operations and interval valued extensions are available for the vector field of the differential equation and its derivatives (up to a specified order) then the resulting computer generated existence proofs and error bounds are mathematically rigorous (we use the MATLAB package INTLAB \cite{Ru99a}  to obtain  rigorous bounds).  However, in light of the additional computational cost of interval arithmetic as compared to floating point, if complete rigor is not needed  then the method can just as easily be used with traditional floating point operations to obtain approximate error bounds.

In Section \ref{sec:newton} we fix some notation and discuss a general strategy for a posteriori existence proofs via the Newton-Kantorovich theorem. For comparison, we outline the method of \cite{Kedem} in Section \ref{sec:IVP_method}.  Development of our method begins in Section  \ref{sec:gen} and continues in subsequent sections. Background on the Green's function of a BVP is given in Section \ref{sec:greens}. The method is tested on several sample problems in Section \ref{sec:num_tests}.

\section{A Posteriori Existence Proofs via Newton-Kantorovich}\label{sec:newton}

Consider a general non-linear two point BVP on $[0,1]$,
\begin{equation}\label{nonlin_bvp}
y^\prime(t)=f(t,y(t)),\hspace{2mm} g(y(0),y(1))=0,
\end{equation}
 where $f:\mathbb{R}^{n+1}\rightarrow\mathbb{R}$ is continuous, differentiable in $y$, $D_yf$ is continuous on $\mathbb{R}^{n+1}$, and $g:\mathbb{R}^{2n}\rightarrow\mathbb{R}^n$ is $C^1$.  In particular, local existence and uniqueness of the IVP is guaranteed.

In order to apply functional analytic methods to this problem it is convenient to follow \cite{Kedem} and define the Banach spaces
\begin{equation}\label{X_def}
X_1=C([0,1],\mathbb{R}^n), \hspace{2mm} C_0([0,1],\mathbb{R}^n)\equiv \{y\in X_1:y(0)=0\},\hspace{2mm} X_2=C_0([0,1],\mathbb{R}^n)\times\mathbb{R}^n
\end{equation}
where the $\ell^\infty$-norm is used on $\mathbb{R}^n$, denoted by $|\cdot|$, the sup-norm $\|y\|\equiv \sup_{[0,1]}|y(t)|$ is used on the function spaces, and  the norm $\|(y,b)\|=\max\{\|y\|,|b|\}$ is used on $X_2$.

 In this setting, the BVP (\ref{nonlin_bvp}) can be put in a form more amenable to analysis as follows.  Define the map $G:X_1\rightarrow X_2$
\begin{equation}\label{G_def}
G[y](t)=\left(y(t)-y(0)-\int_0^tf(s,y(s))ds,g(y(0),y(1))\right).
\end{equation}
Solutions of the BVP are in one to one correspondence with zeros of $G$. Its first component, valued in $C_0([0,1],\mathbb{R}^n)$, is denoted by $G^1$ and its second component, valued in $\mathbb{R}^n$, by $G^2$. As shown in~\cite{Kedem}, $G$ is Fr\'{e}chet-$C^1$ with derivative  at $y_0\in X_1$ given by
\begin{align}\label{DF_formula}
DG^1(y_0)[v](t)=&v(t)-v(0)-\int_0^tD_y f(s,y_0(s))v(s)ds,\\
DG^2(y_0)[v](t)=&D_{y_1}g(y_0(0),y_0(1))v(0)+D_{y_2}g(y_0(0),y_0(1))v(1).\notag
\end{align}

Having reformulated the problem in terms of finding zeros of a $C^1$ map between Banach spaces, a natural tool is the Newton-Kantorovich theorem; see  \ref{newton_th} for a precise statement. Invertibility of $DG(y_0)$ along with a bound on the norm of its inverse are required to apply the theorem and these are typically the most difficult ingredients to obtain in practice.  Once this has been done the remainder of the estimates are relatively straightforward and have been discussed  in, for example, \cite{Kedem}. For this reason we will predominately focus on the invertiblility of $DG(y_0)$ i.e. on a posteriori methods for proving the existence and uniqueness to solutions of linear BVPs.

\section{IVP Fundamental Solution Method}\label{sec:IVP_method}

Motivated by the previous section, consider a linear BVP on $[0,1]$ as given by the bounded linear operator $F:X_1\rightarrow X_2$ defined by the formula (\ref{lin_bvp}),   where $A(t)$ is a continuous $\mathbb{R}^{n\times n}$ valued map.

In Theorem 1 of \cite{Kedem}, the author shows how the solvability of $F[v]=(r,w)$ is related to the matrix of fundamental solutions, defined as the solution to the initial value problem (IVP)
\begin{equation*}
Y^\prime(t)=A(t)Y(t),\hspace{2mm}Y(0)=I.
\end{equation*}
In our notation, this result is the following.
\begin{theorem}\label{thm:ivp}
Let $A(s)$ and $r(t)$ be continuous matrix and vector valued respectively with $r(0)=0$ and consider the two point BVP
\begin{equation*}
v(t)=v(0)+\int_0^t A(s)v(s)ds+r(t),\hspace{2mm} B_0v(0)+B_1v(1)=w.
\end{equation*}
Let $Y(t)$ be the corresponding matrix of fundamental solutions and define $R=B_0+B_1Y(1)$.  Then the BVP has a unique solution iff $R$ is nonsingular and the solution is given by 
\begin{align*}
v(t)=&Y(t)\bigg[R^{-1}B_0\int_0^tY^{-1}(s)A(s)r(s)ds\\
&+(R^{-1}B_0-I)\int_t^1Y^{-1}(s)A(s)r(s)ds-R^{-1}(B_1r(1)-w)\bigg]+r(t).\notag
\end{align*}
\end{theorem}

It was also shown in \cite{Kedem} how to apply this theorem to rigorously prove solvability of the BVP and derive guaranteed error bounds using numerically constructed approximations to $Y(t)$ and $Y^{-1}(t)$.  Because of the use of $Y$, this will be called the IVP fundamental solution method or simply the IVP method.  

  A limitation of this method is that the stability properties of an IVP take center stage; the quantities $|Y(t)|$ and $|Y^{-1}(t)|$ feature critically in the  error estimates derived in \cite{Kedem}.  Unfortunately there are situations in which a  BVP of interest is stable whereas the corresponding IVP is not.  In such situations the bounds obtained by the IVP fundamental solution method will be very pessimistic due to the large size of $|Y(t)|$ and/or $|Y^{-1}(t)|$.  

 As discussed in \cite{ascher_petzold}, perhaps the simplest family of test problems that exhibits these features is 
\begin{equation*}\label{test_prob}
y^{\prime\prime}(t)=y(t),\hspace{2mm} y(0)=1,\hspace{2mm} y(b)=0,
\end{equation*} 
where $b>0$.  We convert this to the equivalent system
\begin{equation*} 
y_1^\prime(t)=by_2(t),\hspace{2mm} y_2^\prime(t)=by_1(t),\hspace{2mm} y_1(0)=1,\hspace{2mm} y_1(1)=0.
\end{equation*} 
 The corresponding IVP fundamental solution matrix and its inverse are
\begin{equation*}
Y(t)=\left( \begin{array}{cc}
\cosh(bt) & \sinh(bt)  \\
\sinh(bt) & \cosh(bt)\\
\end{array} \right),\hspace{2mm}
Y^{-1}(t)=\left( \begin{array}{cc}
\cosh(bt) & -\sinh(bt)  \\
-\sinh(bt) & \cosh(bt)\\
\end{array} \right).
\end{equation*}
 These have exponentially increasing norms as a function of $b$.  This is in spite of the fact that the value and derivative of the exact solution 
\begin{equation*}
y(t)=\cosh(bt)-\cosh(b)\sinh(bt)/\sinh(b)
\end{equation*}
 have $\mathcal{O}(1)$ norm.  

\begin{comment}
$y_1(t/b)$ solves the original system:
\begin{align}
 y_1(t/b)''=y_1''(t/b)/b^2=by_2'(t/b)/b^2=b^2y_1(t/b)/b^2=y_1(t/b).
\end{align}

\begin{align}
Y'(t)=b\left( \begin{array}{cc}
\sinh(bt) & \cosh(bt)  \\
\cosh(bt) & \sinh(bt)\\
\end{array} \right)
\end{align}
satisfies the ode.  The formula for $Y^{-1}$ is correct by  hyperbolic trig identity.
\end{comment}

Even if one tries to avoid separately bounding the norms of $Y$ and $Y^{-1}$ and take advantage of the cancellation that can occur in $Y(t)R^{-1}B_0Y^{-1}(s)$ and $Y(t)(R^{-1}B_0-I)Y^{-1}(s)$, the above example is sufficient to show that this cancellation does not in general occur over the entire $(t,s)$ domain. For example,
\begin{equation*}
(Y(t)R^{-1}B_0Y^{-1}(s))^1_1=(\cosh(bt)-\cosh(b)\sinh(bt)/\sinh(b))\cosh(bs)
\end{equation*}
which diverges at $t=0$, $s=1$ as $b\rightarrow\infty$. Even if cancellation does occur, it is difficult to compute $Y(t)$ and $Y^{-1}(s)$ to high enough accuracy that one could take advantage of such cancellation.

In summary, while the  IVP fundamental solution method suffices for problems whose IVP is stable, if one is interested in BVPs whose corresponding IVP is unstable then an alternative method is needed.  Our contribution is designed to address such situations by utilizing the BVP fundamental solution and Green's function in place of the IVP fundamental solution.

\section{A More General Framework}\label{sec:gen}
We are now in position to begin discussing an alternative method for producing a posteriori existence proofs and error bounds based on the BVP fundamental solution.  This method will be called the BVP fundamental solution method, or BVP method for short, in contrast  to the IVP method.  The BVP fundamental solution is introduced in the following section, but first we discuss the function spaces that will be used in our analysis.

The input to any a posteriori method consists of a certain (numerically generated) approximate solution.  For us this will consist of function values on a mesh $0=t_0<...<t_N=1$.  Before error bounds can be produced, these need to be extended to an approximate solution on the entire interval $[0,1]$.  While one could work in spaces of continuous functions  by interpolating the values at the nodes, as done in the IVP method, we will find things more convenient (and the resulting error bounds more transparent) when this assumption is removed, both from the solution space and the matrix $A(t)$. In this section we outline the slightly generalized framework in which the equations will be formulated.

To this end, fix a mesh, $t_i$, as above and let $Y_1$ be the space of functions that are piecewise continuous on this mesh.  More precisely, $Y_1$ is the direct sum
\begin{equation}\label{Y1_def}
Y_1=\bigoplus_{j=1}^NC([t_{j-1},t_j],\mathbb{R}^n).
\end{equation} 
On each factor the norm will be taken to be a weighted $\ell^\infty$-norm
\begin{equation*}
\|(y_1,...,y_N)\|_W\equiv \max_{j=1,...,N}\|y_j\|_j, \hspace{2mm} \|y_j\|_j\equiv\sup_{t\in[t_{j-1},t_j]}|y_j(t)|_W,\hspace{2mm} |v|_W\equiv |W v|_{\ell^\infty}
\end{equation*}
where the weight, $W$, is some invertible matrix.  In practice, we will take $W$ to be diagonal and choose the diagonal elements to balance the magnitude of the errors in the different components of the solution. One could consider the generalization where the weight is allowed to vary between subintervals but this will not be explored here.

$Y_1$ is a Banach space and the maps that send any $y\in Y_1$ to its value at any $t\in\{0,1\}\cup_{j=1}^N(t_{j-1},t_j)$ or at $t_j^{\pm}$ ($+$ denots limit from above, $-$  limit from below) are well defined bounded linear maps. Elements of $Y_1$ can be identified with elements of $L^\infty$ and  integration or differentiation (of piecewise $C^1$) elements of $Y_1$ will be defined using this identification. 

Mirroring the continuous case, (\ref{X_def}),  define a second space
\begin{equation}\label{Y2_def}
Y_2=\{v\in Y_1: v(0)=0\}\oplus\mathbb{R}^n.
\end{equation}
This is a Banach space under the norm $\|(v,b)\|_{W_1,W_2}=\max\{\|v\|_{W_1},|b|_{W_2}\}$ where $W_i$ are any invertible matrices. Finally, note that the $X_i$'s are closed subspaces of the $Y_i$'s.  For simplicity we will take $W_1=W_2=W$.  This will allow us to drop the subscript $W$'s from the above norms in the following sections.

Having specified the functional analytic framework, the definition of nonlinear BVPs can be extended to this setting by defining $G:Y_1\rightarrow Y_2$ by the same formula as in (\ref{G_def}).  Note that any zero of this map is automatically continuous, and so solutions of the generalized problem are still in 1-1 correspondence with solutions of the original BVP.

The extended operator $G$ is still Fr\'echet-$C^1$, with the same formula for the derivative (\ref{DF_formula}).  Note, however, that $D_y f(s,y_0(s))$ is no longer continuous as a function of $s$ for a general $y_0\in Y_1$. Therefore, in the interest of applying Newton-Kantorovich to prove existence, we are lead to consider bounded linear maps $F:Y_1\rightarrow Y_2$ defined as in (\ref{lin_bvp}), where the assumptions on $A(s)$ are relaxed to allow for piecewise continuity on the given mesh.

\section{Green's Function for Linear BVPs}\label{sec:greens}
We are now in position to discuss the BVP fundamental solution, on which our method will be based. Let $A:[0,1]\rightarrow \mathbb{R}^{n\times n}$ be piecewise continuous on the mesh $t_j$ (i.e. continuous, except for possible jumps at the $t_j$'s).  A BVP fundamental solution is defined as a continuous map $\Phi:[0,1]\rightarrow\mathbb{R}^{n\times n}$ consisting of invertible matrices that is piecewise $C^1$ on the mesh $t_j$ and solves
\begin{equation*}
\Phi^\prime=A\Phi,\hspace{2mm} B_0\Phi(0)+B_1\Phi(1)=I.
\end{equation*}
 Analogously to the IVP fundamental solution, if a BVP fundamental solution exists then it provides an inverse to the BVP operator $F$ through the Green's function, as defined in the following theorem.
\begin{theorem}\label{lin_BVP_solution}
If a BVP fundamental solution $\Phi$ exists then it is unique and the unique solution $v\in Y_1$ to
\begin{equation*}
v(t)-v(0)-\int_0^t A(s)v(s)ds=r(t),\hspace{2mm} B_0v(0)+B_1v(1)=w,
\end{equation*}
where $(r,w)\in Y_2$,
can be written
\begin{align*}
v(t)=&\Phi(t)(w-B_1r(1))+r(t)+\int_0^1 G(t,s)A(s)r(s)ds
\end{align*}
where the  Green's function is defined as
\begin{align}\label{greens_def}
G(t,s)=\begin{cases}
\Phi(t)B_0\Phi(0)\Phi^{-1}(s), & \text{if }{s\leq t} \\
-\Phi(t)B_1\Phi(1)\Phi^{-1}(s), & \text{if }{ s>t}.
 \end{cases}
\end{align}
Therefore $\Phi$ provides us with an inverse of the operator $F:Y_1\to Y_2$ (given by the formula (\ref{lin_bvp})),\begin{equation}\label{F_inv_formula}
F^{-1}:(r,w)\rightarrow\Phi(t)(w-B_1r(1))+r(t)+\int_0^1 G(t,s)A(s)r(s)ds.
\end{equation}
\end{theorem}
Note that the Green's function (\ref{greens_def}) is not to be confused with the notation for the operator associated with a nonlinear BVP (\ref{G_def}).  The meaning will be clear from the context.

Theorem \ref{lin_BVP_solution} is essentially found in \cite{ascher_petzold}.  The only complications beyond the presentation there is that $A(t)$ is only piecewise continuous on the mesh and $r(t)$ must be taken to be a piecewise continuous function on the mesh that vanishes at $t=0$, and not necessarily a function of the form
\begin{equation*}
r(t)=\int_0^tq(s)ds
\end{equation*}
for some continuous $q$, as would be the case when transforming from the differential to the integral form of the BVP.  One can directly verify the formula by composing (\ref{F_inv_formula}) with (\ref{lin_bvp}) and vice versa.

\section{Proving Invertibility}
We now suppose that we are given $C^1$ maps $\tilde \Phi_j:[t_{j-1},t_j]\rightarrow\mathbb{R}^{n\times n}$, $\tilde G_{i,j}:[t_{i-1},t_i]\times[t_{j-1},t_j]\rightarrow\mathbb{R}^{n\times n}$, $i\neq j$, 
\begin{align}
\tilde G^+_{i,i}:\{(t,s)\in [t_{i-1},t_i]\times[t_{i-1},t_i]:s\geq t\} \rightarrow\mathbb{R}^{n\times n},
\end{align}
and
\begin{align}
\tilde G^-_{i,i}:\{(t,s)\in [t_{i-1},t_i]\times[t_{i-1},t_i]:s\leq t\} \rightarrow\mathbb{R}^{n\times n}
\end{align}
 which are thought of  approximations to the fundamental BVP solution and Green's function on their domains (but without any a priori knowledge that an exact $\Phi$ actually exists) which were obtained by (non-rigorous) numerical means. We will construct $\tilde G$ using $\tilde\Phi$ and an approximation to its inverse by replacing the corresponding exact quantities in the definition of the Green's function (\ref{greens_def}), but this is not strictly necessary.

Our strategy for proving the existence of $F^{-1}:Y_2\rightarrow Y_1$ is to mimic the formula (\ref{F_inv_formula}) for $F^{-1}:X_2\rightarrow X_1$ in terms of the fundamental solution  but using the approximate fundamental solution and Green's function instead. More precisely, let $\tilde G_{i,i}(t,s)= \tilde G_{i,i}^-(t,s)1_{s\leq t}+\tilde G_{i,i}^+(t,s) 1_{s>t}$ and define the bounded linear map $H:Y_2\to Y_1$,
\begin{align}
H[r,w]_i(t)=& \tilde\Phi_i(t)(w-B_1r(1))+r_i(t)+\sum_{j=1}^N\int_{t_{j-1}}^{t_j} \tilde G_{i,j}(t,s) A_j(s) r_j(s) ds.
\end{align}
where the subscript $i$ refers to the component on the interval $[t_{i-1},t_i]$.

$H$ will be used to find conditions under which we can prove the invertiblility of $F$. This will rely on the following well known result concerning perturbations of the identity, see for example \cite{folland}.
\begin{lemma}\label{I_perturb_lemma}
Let $A:X\rightarrow X$ be a bounded linear map on a Banach space.  If $\|I-A\|<1$ then $A$ has a bounded inverse that satisfies
\begin{equation*}
\|A^{-1}\|\leq\frac{1}{1-\|I-A\|}.
\end{equation*}
\end{lemma}

In fact, a slight generalization of this result will be needed, given below.
\begin{lemma}\label{inverse_lemma}
Let $A:X\rightarrow Y$, $B:Y\rightarrow X$ be bounded linear maps where $X,Y$ are Banach spaces. If $\|I-AB\|\leq\alpha$ for $\alpha<1$ where $I$ is the identity on $Y$ then $A$ is surjective.   If $A$ is also injective then $A$, $B$ are invertible, 
\begin{equation*}
\|A^{-1}\|\leq\frac{\|B\|}{1-\alpha} \text{ and } \|A^{-1}-B\|\leq \frac{\alpha \|B\|}{1-\alpha}.
\end{equation*}
\end{lemma}
\begin{proof}
By the above lemma, $AB$ has a bounded inverse with
\begin{equation*}
\|(AB)^{-1}\|\leq\frac{1}{1-\|I-AB\|}\leq \frac{1}{1-\alpha}.
\end{equation*}
 In particular $A$ is surjective.  It is injective by assumption so  by the open mapping theorem it has a bounded inverse. This implies $B$ has a bounded inverse as well and
\begin{equation*}
\|A^{-1}\|=\|BB^{-1}A^{-1}\|=\|B(AB)^{-1}\|\leq \frac{\|B\|}{1-\alpha},
\end{equation*}
\begin{equation*} 
\|A^{-1}-B\|=\|A^{-1}(I-AB)\|\leq \|A^{-1}\|\|I-AB\|\leq \frac{\alpha \|B\|}{1-\alpha}.
\end{equation*}
\end{proof}

\subsection{Applying the Fredholm Alternative}\label{sec:fred_alt}
In finite dimensions and when $X$ and $Y$ have the same dimension then the hypothesis of injectivity of $A$ in Lemma \ref{inverse_lemma} is not needed but in infinite dimensions it is required in general. However, when we have the appropriate element of compactness, the situation in infinite dimensions mirrors the finite dimensional case.  More specifically,  recall one consequence of the Fredholm alternative, see for example \cite{eidelman2004functional}.
\begin{theorem}
Let $X$ be a Banach space and $K:X\rightarrow X$ be a compact operator.  Then $I-K$ is injective iff it is surjective.
\end{theorem}

If a Fredholm alternative-like condition holds for $F$ then surjectivity will imply injectivity, and hence $\|I-FH\|<1$ will imply the invertiblity of $F$.  In addition, this will provide a bound on $F^{-1}$ per Lemma \ref{inverse_lemma}.  These are the two ingredients that are needed in order to utilize the Newton-Kantorovich theorem.  With slight modifications of the above we could just as well work with $I-HF$ but find this less appealing, as discussed in  \ref{app:I-HF}.

In order to more easily apply the Fredholm alternative, we reformulate the BVP operator $F$ in terms of a map from a single Banach space to itself. Define $\tilde F:Y_2\rightarrow Y_2$ by
\begin{align*}
\tilde F[v,w](t)=&\left(v(t)-\int_0^t A(s)(w+v(s))ds,B_0w+B_1(w+v(1))\right)\\
=&(v(t),w)-\left(\int_0^t A(s)(w+v(s))ds,(I-B_0-B_1)w-B_1v(1)\right)\\
\equiv& (I-K)[v,w](t).
\end{align*}
We have $F[v]=\tilde F[v-v(0),v(0)]$ and so $F$ is surjective or injective if and only if $\tilde F$ is.  

\begin{comment}
Let $v\in Y_1$.  Then $(v-v(0),v(0))\in Y_2$ and 
\begin{align}
&\tilde F[v-v(0),v(0)]=\left(v(t)-v(0)-\int_0^t A(s)(v(0)+v(s)-v(0))ds,B_0v(0)+B_1(v(0)+v(1)-v(0))\right)\\
=&\left(v(t)-v(0)-\int_0^t A(s)v(s)ds,B_0v(0)+B_1v(1)\right)\\
=&F[v](t).
\end{align}

\end{comment}

The second component of $K$ has finite rank and the first is an integral operator with bounded kernel over a bounded domain and whose image consists of continuous functions, hence the compactness of $K$ follows from an application of the Arzela-Ascoli theorem.  This implies that the Fredholm alternative applies to $\tilde F$ and hence, showing that $\|I-FH\|\leq\alpha<1$ is sufficient to prove that $F$ has a bounded inverse with
\begin{equation}\label{F_inv_bound}
\|F^{-1}\|\leq\frac{\|H\|}{1-\alpha},\hspace{2mm} \|F^{-1}-H\|\leq  \frac{\alpha \|H\|}{1-\alpha}.
\end{equation}
We will now discuss how to bound $\|I-FH\|$ and $\|H\|$.

\begin{comment}
Here we prove compactness of $K$ in detail:\\
\begin{align}
K[v,w](t)=\left(\int_0^t A(s)(w+v(s))ds,(I-B_0-B_1)w-B_1v(1)\right)
\end{align}
We need to show that the image of a bounded set has compact closure.  This holds if for every bounded sequence $(v_n,w_n)\in Y_2$, $K[v_n,w_n]$ has a convergence subsequence.\\

\begin{align}
K[v_n,w_n]_2=(I-B_0-B_1)w_n-B_1v_n(1)
\end{align}
is a bounded sequence in $\mathbb{R}^n$, hence is has a convergence subsequence labeled by $n_j$.  Replacing $(v_n,w_n)$ with $(v_{n_j},w_{n_j})$, it now suffices to consider the first component,
\begin{align}
K[v_n,w_n]_{1,i}(t)=\int_0^t A_i(s)(w_n+(v_n)_i(s))ds.
\end{align}
$i=1,...,N$.  For each $i$, these are continuous functions on the compact set $[t_{i-1},t_i]$.  So by Arzela Ascoli if we can show equicontinuity and pointwise boundedness then each of these sequences has a convergent subsequence.  By induction, we can obtain a  subsequence of all of these (i.e. for all $i$) that is also convergent and so we will be done.\\

We have
\begin{align}
|K[v_n,w_n]_{1,i}(t)|\leq \|A_i\|_\infty (\sup_n\|w_n\|+\sup_n\|v_n\|_\infty) t<\infty,
\end{align}
therefore pointwise boundedness holds.  Also,
\begin{align}
&|K[v_n,w_n]_{1,i}(t)-K[v_n,w_n]_{1,i}(s)|\leq \int_{\min\{s,t\}}^{\max\{s,t\}}|A_i(r)(w_n+(v_n)_i(r))dr|\\
\leq &  \|A_i\|_\infty (\sup_n\|w_n\|+\sup_n\|v_n\|_\infty) |t-s|.
\end{align}
So equicontinuity holds. This completes the proof of compactness.

\end{comment}

\subsection{Bounding $\|I-FH\|$}
In   \ref{app:I-FH}  a formula for $I-FH$ is derived.  For each $t\not\in\{ t_i\}_{i=0}^N$ let $j_t=\max\{j:t_j<t\}$.  The two components of $I-FH$ are given by
\begin{align}\label{I-FH_components}
&(I-FH)[r,w]_1(t)\\
=&-\int_0^t\tilde \Phi^\prime(s)-A(s)\tilde\Phi(s)ds(w-B_1r(1))+\int_0^t(\tilde G(z^-,z)-\tilde G(z^+,z)+I)A(z)r(z)dz\notag\\
&+\sum_{j=1}^{j_t}(\tilde\Phi(t_j^-)-\tilde\Phi(t_{j}^+))(w-B_1r(1))+\int_0^1\left[\sum_{j=1}^{j_t}(\tilde G(t_j^-,z)-\tilde G(t_{j}^+,z))\right]A(z)r(z)dz\notag\notag\\
&-\int_0^t\int_0^1(\partial_s \tilde G(s,z)-A(s) \tilde G(s,z))A(z)r(z)dzds,\notag\\
&(I-FH)[r,w]_2\notag\\
=&[I-B_0\tilde\Phi(0)-B_1\tilde\Phi(1)](w-B_1r(1))-\int_0^1 \left(B_0\tilde G(0,s)+B_1\tilde G(1,s)\right)A(s)r(s)ds\notag
\end{align}
where superscripts $+$ and $-$ denote the limits from above and below respectively.  For completeness, the analogous expression for $I-HF$ is given in  \ref{app:I-HF}.

The most straightforward way to use (\ref{I-FH_components}) to bound the norm of $I-FH$ is to define $\tilde G$ using (\ref{greens_def}), replacing the exact quantities with $\tilde\Phi$ and its approximate inverse, and then use the submultiplicative property of induced matrix norms and bound each of the quantities in the Green's function individually.  This leads to an estimate whose operation count scales linearly in the number of subintervals, $N$. If $|\Phi|$ and $|\Phi^{-1}|$ are not too large over $[0,1]$ then this is a good strategy.  However, this is often not the case, even in situations where the BVP is stable.  We illustrate this phenomenon with the same test problem introduced in (\ref{test_prob}).

\subsection{Behavior of the Test Problem}\label{sec:test_prob}
Consider again the simple test problem, (\ref{test_prob}). The fundamental solution and its inverse are
\begin{equation*}
\Phi(t)=\frac{1}{\sinh(b)}\left( \begin{array}{cc}
\sinh(b(1-t)) & \sinh(bt)  \\
-\cosh(b(1-t)) & \cosh(bt)\\
\end{array} \right),\hspace{2mm}
\Phi^{-1}(t)=\left( \begin{array}{cc}
\cosh(bt) & -\sinh(bt)  \\
\cosh(b(1-t)) & \sinh(b(1-t))\\
\end{array} \right).
\end{equation*}
While $\|\Phi\|$ is bounded, $\|\Phi^{-1}\|$ diverges as $b\rightarrow\infty$.  In fact, both $\|B_0\Phi(0)\Phi^{-1}\|$ and $\|B_1\Phi(1)\Phi^{-1}\|$ diverge as well.

 The general notion of BVP stability only requires the norm of the fundamental solution and of the Green's function to be small, but not the inverse of the fundamental solution. See for example \cite{ascher_petzold} for details. Therefore, in general we must abandon any estimate based on bounding $\|\tilde\Phi\|$ and $\|\tilde\Phi^{-1}\|$ separately and use one based on bounding the entire Green's function. The test problem (\ref{test_prob}) does have a  bounded Green's function 
\begin{equation*}
G(t,s)=\begin{cases}
\frac{1}{\sinh(b)}\left( \begin{array}{cc}
\sinh(b(1-t))\cosh(bs) & -\sinh(b(1-t))\sinh(bs)  \\
-\cosh(b(1-t))\cosh(bs) &\cosh(b(1-t))\sinh(bs)\\
\end{array} \right), & \text{if }{s\leq t} \\
-\frac{1}{\sinh(b)}\left( \begin{array}{cc}
\sinh(bt)\cosh(b(1-s)) & \sinh(bt)\sinh(b(1-s))  \\
 \cosh(bt)\cosh(b(1-s))& \cosh(bt)\sinh(b(1-s))\\
\end{array} \right), & \text{if }{ s>t}.\\
 \end{cases}
\end{equation*}

We emphasize that $\|G\|=\mathcal{O}(1)$ does not arise from a large cancellation of the computed results.  Rather, at each endpoint the boundary conditions annihilate the growing modes of $\Phi^{-1}$.  This suggests that we should not run into the same degree of numerical cancellation issues using the BVP method as opposed to the IVP method.

\section{A Sharper Bound}\label{sec:sharper_bound}
In this section we show how to compute a sharper bound on $\|I-FH\|$ by bounding the Green's function as a whole.  Bounding the second component of $I-FH$ is straightforward so we focus on the first.  

To this point the only assumption made about $\tilde\Phi$ and $\tilde G$ is that they are piecewise $C^1$.  In practice, a non-rigorous numerical algorithm will provide us with approximations $\hat\Phi_0$ to $\Phi(0)$,  $\hat\Phi_1$ to $\Phi(1)$, as well as   $\tilde\Phi_j$ and $\tilde G_{i,j}$, approximations to $\Phi$ and $G$ at the centers of the intervals $[t_{i-1},t_i]$ and rectangles $[t_{i-1},t_i]\times [t_{j-1},t_j]$ respectively.   

In practice, we obtain the the $\tilde\Phi_j$ by numerically solving the BVP and, for $i\neq j$, define
\begin{equation}\label{G_ij_def}
\tilde G_{i,j}=\begin{cases}
\tilde\Phi_iB_0\hat\Phi_0\Psi_j, & \text{if }{j< i} \\
-\tilde\Phi_iB_1\hat\Phi_1\Psi_j, & \text{if }{ j> i}
 \end{cases}
\end{equation}
where $\Psi_j$ is a numerical approximation to $\Phi_j^{-1}$ and the multiplications are carried out in {\em computer arithmetic}. On the diagonal $i=j$ two different versions of $\tilde G_{i,i}$ are needed,
\begin{equation}\label{G_pm_def}
 \tilde G^-_{i,i}=\frac{1}{2}\tilde\Phi_i\left(B_0\hat\Phi_0-B_1\hat\Phi_1\right)\Psi_i+\frac{1}{2}I,\hspace{2mm} \tilde G^+_{i,i}=\frac{1}{2}\tilde\Phi_i\left(B_0\hat\Phi_0-B_1\hat\Phi_1\right)\Psi_i-\frac{1}{2} I,
\end{equation}
 corresponding to the lower and upper triangular halves respectively.  Note that these reduce to the correct expressions when the exact BVP fundamental solution and its inverse are used and the calculations are carried out in exact arithmetic. Also note that the jump condition $\tilde G^-_{i,i}-\tilde G^+_{i,i}=I$  holds in exact arithmetic.

\begin{comment}
Here we verify the claim that these reduce to the correct expressions when the exact BVP solution is used:
\begin{align}
& \frac{1}{2}\Phi_i\left(B_0\Phi(0)-B_1\Phi(1)\right)\Phi^{-1}_i+\frac{1}{2}I\\
=& \frac{1}{2}\Phi_i\left(B_0\Phi(0)-B_1\Phi(1)+I\right)\Phi^{-1}_i\\
=& \frac{1}{2}\Phi_i\left(B_0\Phi(0)-B_1\Phi(1)+B_0\Phi(0)+B_1\Phi(1)\right)\Phi^{-1}_i\\
=&\Phi_iB_0\Phi(0)\Phi^{-1}_i\\
\end{align}
and
\begin{align}
&\frac{1}{2}\Phi_i\left(B_0\Phi(0)-B_1\Phi(1)\right)\Phi^{-1}_i-\frac{1}{2} I\\
=&\frac{1}{2}\Phi_i\left(B_0\Phi(0)-B_1\Phi(1)-I\right)\Phi^{-1}_i\\
=&\frac{1}{2}\Phi_i\left(B_0\Phi(0)-B_1\Phi(1)-(B_0\Phi(0)+B_1\Phi(1))\right)\Phi^{-1}_i\\
=&-\Phi_iB_1\Phi(1)\Phi^{-1}_i.
\end{align}
\end{comment}

The bounds we will derive below will not depend on the details of how one obtains approximations to the fundamental solution and Green's function at the centers of the regions.  Irrespective of how they are constructed, given  values $\tilde\Phi_j$, $\tilde G_{i,j}$, and $\tilde G_{i,i}^{\pm}$, we construct the functions  $\tilde G_{i,j}(t,s)$ and $\tilde \Phi_i(t)$  on each subinterval or rectangle of the mesh to be the polynomials that satisfy the equations
\begin{equation*}
\partial_t  G(t,s)=A(t) G(t,s),\hspace{2mm} \partial_s  G(t,s)= -G(t,s)A(s),\hspace{2mm}G(t_{i-1/2},t_{j-1/2})=\tilde G_{i,j}
\end{equation*}
 and
\begin{equation*}
\Phi^\prime=A\Phi,\hspace{2mm} \Phi(t_{j-1/2})=\tilde\Phi_{j}
\end{equation*} 
up to some order $m-1$, $m\geq 1$ (here we assume each $A_i(t)$ is $C^m$).  Formulas for the coefficients and the ODE residuals are given in  \ref{app:taylor}.    Again, on the rectangle with $i=j$ different versions of $\tilde G$ are needed, one corresponding to the upper triangle and one to the lower, constructed in this manner from $\tilde G^+_{i,i}$ and $\tilde G^-_{i,i}$ respectively. The resulting functions $\tilde\Phi_i$, $\tilde G_{i,j}$, and $\tilde G_{i,i}^\pm$ can be pieced together into piecewise smooth functions, $\tilde \Phi$ and $\tilde G$ on $[0,1]$ and $[0,1]\times[0,1]$.  This is the reason we generalized the solution space in Section \ref{sec:gen} to allow for jump discontinuities.

In the absence of jump discontinuities at the mesh points, this method reduces to a piecewise Taylor model IVP method for $\Phi$, see for example \cite{Nedialkov199921}.  While well suited for initial value problems, such methods perform extremely poorly for BVPs with an unstable IVP such as the examples in Section \ref{sec:num_tests}. Therefore, having accurate $\tilde\Phi_{j}$'s as inputs is crucial to the success of the method; the Taylor series simply ensure sufficiently small error over the interior of each subinterval of the mesh.

 Using (\ref{I-FH_components}) and \ref{app:taylor} we can derive the following bound on $I-FH$.
\begin{align}\label{I-FH_bound}
\|I-FH\|= \max\{\|(I-FH)_1\|,|(I-FH)_2|\}
\end{align}
where
\begin{align}
&\|(I-FH)_1\|\label{I-FH_bound2_1}\\
\leq&(1+|B_1|)\sum_{j=1}^{N-1} \left|P_j(t_j)\tilde\Phi_j -P_{j+1}(t_j) \tilde\Phi_{j+1} \right|+\frac{1+|B_1|}{(m+1)2^m}\sum_{j=1}^{N}h_j^{m+1}\sup_s| R_j(s)\tilde\Phi_j| \notag\\
&+\sum_{j=1}^{N}h_j \sup_z\left|P_j(z)(I-\tilde G_{j,j}^-+\tilde G_{j,j}^+)Q_j(z)A_j(z) \right| +  \frac{1}{(m+2)2^{m+1}}\sum_{j=1}^{N}h_j^{m+2}\sup_z| \tilde R_j(z)A_j(z)|\notag\\
&+\frac{1}{(m+1)2^m}\sum_{j=1}^{N}  h_j^{m+1}\sum_{k=1}^N h_k\max\{\sup_{s,z}|R_j(s)\tilde G_{j,k}^+Q_k(z)A_k(z)|,\sup_{s,z}|R_j(s)\tilde G_{j,k}^-Q_k(z)A_k(z)|\} \notag\\
&+\sum_{j=1}^{N-1} \sum_{k=1}^N h_k \sup_z\left| ( P_j(t_j)\tilde G_{j,k}^--P_{j+1}(t_j) \tilde G_{j+1,k}^+)Q_k(z) A_k(z) \right|\notag
\end{align}
and
\begin{align}
&|(I-FH)_2|\leq(1+ |B_1|)|I-B_0 P_1(0)\tilde\Phi_1-B_1 P_N(1)\tilde\Phi_N|\label{I-FH_bound2_2}\\
&+\sum_{j=1}^Nh_j\sup_s|(B_0 P_1(0)\tilde G_{1,j}^+ +B_1 P_N(1)\tilde G_{N,j}^- )Q_j(s)A_j(s)| .\notag
\end{align}

Here we have defined $\tilde G^{\pm}_{i,j}=\tilde G_{i,j}$ for $i\neq j$, $h_i=t_i-t_{i-1}$, the subscripts $i$  denote the functions restricted to $[t_{i-1},t_i]$,  $P_i(t)$ and $Q_i(t)$ are defined in (\ref{Pj_def}) and (\ref{Qj_def}) respectively, and $\tilde R_i(t)$ is defined in (\ref{PQ_I_resid}). 

To bound $R_i(t)$, the ODE residual (\ref{P_ODE_resid}), on $[t_{i-1},t_i]$ a means for analytically bounding the error term in the Taylor series expansion of $A_i(t)$ is needed.  This can be achieved from an interval matrix valued extension for $A_i(t)$ and its derivatives.

To use (\ref{F_inv_bound}) to bound the norm of $F^{-1}$ a machine computable bound on $\|H\|$ is needed.
\begin{align}\label{H_bound}
\|H\|\leq&  \max_i\bigg\{1+(1 +|B_1|)\sup_t|P_i(t) \tilde\Phi_i|\\
&+\sum_{j=1}^N h_j \max\{\sup_{t,s}|P_i(t)\tilde G_{i,j}^-Q_j(s)A_j(s)|,\sup_{t,s}|P_i(t)\tilde G_{i,j}^+Q_j(s)A_j(s)|\}\bigg\} .\notag
\end{align}

Using interval arithmetic, interval matrix enclosures of $P_i(t)$, $Q_j(t)$, $A_j(t)$ etc. can be computed and used to bound the matrix products in \req{I-FH_bound2_1}, \req{I-FH_bound2_2}, and \req{H_bound}. This leads to an algorithm whose complexity has   $\mathcal{O}(n^3N^2)$  as its leading order term in $N$, where $n$ is the dimension of the ODE system.   In particular, it does not depend on $m$, and so from the point of view of complexity, the degree of $\tilde\Phi$ can be set as high as is necessary for the Taylor series errors over the subintervals to be negligible, making the boundary conditions and jump terms the main sources of error. Parallelizing the computation of these bounds is straightforward.

Our main result can be summarized in the following theorem.
\begin{theorem}\label{main_theorem}
Let $Y_1$ and $Y_2$ be defined by (\ref{Y1_def}) and (\ref{Y2_def}) respectively and consider the two point BVP $F[v]=(r,w)$  defined by $F:Y_1\rightarrow Y_2$,
\begin{equation*}
F[v](t)=\left(v(t)-v(0)-\int_0^t A(s)v(s)ds,B_0v(0)+B_1v(1)\right)
\end{equation*}
where $A(t):[0,1]\rightarrow \mathbb{R}^{n\times n}$ is piecewise $C^m$ on the mesh $\{t_i\}_{i=0}^N$ for some $m\geq 1$ and $B_i\in\mathbb{R}^{n\times n}$.

Let the piecewise polynomial functions $\tilde\Phi(t)$ and $\tilde G(t,s)$ be constructed from  $\tilde \Phi_i$ and $\tilde G_{i,j}^{\pm}$ as detailed in \ref{app:taylor}.

  Define $H:Y_2\rightarrow Y_1$ by 
\begin{align*}
H[r,w](t)&=\tilde\Phi(t)(w-B_1r(1))+r(t)+\int_0^1 \tilde  G(t,s)A(s)r(s)ds.
\end{align*}

If $\|I-FH\|\leq\alpha <1$, as computed by (\ref{I-FH_bound}-\ref{I-FH_bound2_2}), then $F^{-1}:Y_2\rightarrow Y_1$ exists,  is bounded, and
\begin{equation*}
\|F^{-1}\|\leq\frac{\|H\|}{1-\alpha}.
\end{equation*}
\end{theorem}

Finally, we note that, if the $|\Psi_i|$ are not too large, then bounds with $O(N)$ complexity in $N$ can be obtained by using (\ref{G_ij_def}) and (\ref{G_pm_def}) to split the $i$ and $j$ dependence into separate terms.

\section{Validated Solutions to  Inhomogeneous Linear BVPs}\label{sec:inhomogeneous}
For  linear BVPs it is useful replace the Newton-Kantorovich theorem with a simpler and more explicit a posteriori error bound. Suppose we have proven $\|I-FH\|\leq\alpha <1$ so that $F$ is invertible.  Let $v$ be the exact solution of $F[v]=(r,w)$ and suppose we are given  an approximate solution $\tilde v\in Y_1$.  Then by (\ref{F_inv_bound}),
\begin{equation}\label{inhomo_err}
\|v-\tilde v\|=\|F^{-1}[(r,w)-F[\tilde v]]\|\leq \|F^{-1}\|\|F[\tilde v]-(r,w)\|\leq\frac{\|H\|}{1-\alpha}\|F[\tilde v]-(r,w)\|.
\end{equation}

 We suppose that on each subinterval $[t_{j-1},t_j]$ of the mesh, $A(t)$ is $C^m$ and $r(t)$ is $C^{m+1}$.  They then have Taylor series with remainders
\begin{align*}
A_i(t)=&\sum_{l=0}^{m-1}A_i^l (t-t_{i-1/2})^l+A^m_i(t)(t-t_{i-1/2})^m\equiv \tilde A_i(t)+A^m_i(t)(t-t_{i-1/2})^m\\
 r_i(t)=&\sum_{l=0}^{m}r_i^l (t-t_{i-1/2})^l+r_i^{m+1}(t)(t-t_{i-1/2})^{m+1}\equiv\tilde r_i(t)+r_i^{m+1}(t)(t-t_{i-1/2})^{m+1}.
\end{align*}
Similar to the construction of $\tilde \Phi$ and $\tilde G$ in \ref{app:taylor}, we suppose that $\tilde v(t)\in Y_1$ is defined from a set of approximations, $\tilde v_i$, to $v(t_{i-1/2})$ as follows:
\begin{equation*}
\tilde  v_j(t)=\sum_{l=0}^m c_j^l(t-t_{j-1/2})^l, \hspace{2mm} c_j^0=\tilde v_{j},\hspace{2mm} c_j^{l}=r_j^{l}+\frac{1}{l}\sum_{i=0}^{l-1}A_j^i c_j^{l-i-1}\text{ for } l=1,...,m.
\end{equation*}

 The residual in (\ref{inhomo_err}) can then be bounded as follows
\begin{align}\label{forward_error}
&\|F[\tilde v]-(r,w)\|\\
=&\max\bigg\{ |B_0\tilde v_1(0)+B_1\tilde v_N(1)-w|,\max_j \sup_{t\in[t_{j-1},t_j]}\left|\int_{t_{j-1}}^{t} \tilde v_j^\prime(s)-\tilde A_j(s)\tilde v_j(s)-\tilde r^\prime_j(s)ds\right.\notag\\
&+\tilde v_j(t_{j-1})-\tilde r_j(t_{j-1})-\tilde v(0)-\sum_{k=1}^{j-1}\int_{t_{k-1}}^{t_k} \tilde A_k(s)\tilde v_k(s)ds-r_j^{m+1}(t)(t-t_{j-1/2})^{m+1}\notag\\
&\left.-\int_{t_{j-1}}^{t} A^m_j(s)\tilde v_j(s)(s-t_{j-1/2})^mds-\sum_{k=1}^{j-1}\int_{t_{k-1}}^{t_k} A^m_k(s)\tilde v_k(s)(s-t_{k-1/2})^mds \right|\bigg\}.\notag
\end{align}

The integrals in \req{forward_error} are either integrals of polynomials, and thus can be integrated exactly, or we can compute interval-vector enclosures of them using the lemmas in \ref{app:int_bound}. We note that the cost is just $\mathcal{O}(N)$ as opposed to the $\mathcal{O}(N^2)$ cost of bounding $\|F^{-1}\|$.  

We also give the more explicit, but generally looser, bound 
\begin{align}
&\|F[\tilde v]-(r,w)\|\\
=&\max\bigg\{ |B_0\tilde v_1(0)+B_1\tilde v_N(1)-w|,\notag\\
&\max_j \sup_{t\in[t_{j-1},t_j]}\left|\sum_{k=1}^{j-1}\left(\tilde v_{k+1}(t_{k})-\tilde v_k(t_k)+\tilde r_k(t_k)-\tilde r_{k+1}(t_{k})\right)-r_j^{m+1}(t)(t-t_{j-1/2})^{m+1}\right.\notag\\
&\left.+\sum_{k=1}^{j-1}\int_{t_{k-1}}^{t_k}\tilde v_k^\prime(s) -A_k(s)\tilde v_k(s)-\tilde r_k^\prime(s)ds+\int_{t_{j-1}}^{t} \tilde v_j^\prime(s)-A_j(s)\tilde v_j(s)-\tilde r_j^\prime(s)ds\right|\bigg\}\\
\leq&\max\bigg\{ |B_0\tilde v_1(0)+B_1\tilde v_N(1)-w|,\label{forward_error2}\\
&\sum_{k=1}^{N-1}\left|\tilde v_{k+1}(t_{k})-\tilde v_k(t_k)+\tilde r_k(t_k)-\tilde r_{k+1}(t_{k})\right|+\max_k\{(h_k/2)^{m+1}\sup_{t}|r_k^{m+1}(t)|\}\notag\\
&+\frac{1}{(m+1)2^{m}} \sum_{k=1}^{N}h_k^{m+1}\sup_t |\hat R_k(t)|\bigg\}\notag
\end{align}
where
\begin{align}
\hat R_k(t)=\sum_{j=0}^{m}(t-t_{k-1/2})^{j}  \sum_{l=0}^{m-j} A_k^{l+j}(t) c_k^{m-l}.
\end{align}

\section{Validated Solutions to Nonlinear BVPs}\label{sec:nonlin}
As mentioned previously, the Newton Kantorovich theorem \req{Newton} allows one to apply Theorem \ref{main_theorem} to prove existence of solutions to nonlinear BVPs.  In this section, we discuss this in more detail.

As before consider a 2-point BVPs on $[0,1]$, \req{nonlin_bvp}, and define the operator $G:Y_1\rightarrow Y_2$ as in \req{G_def}.  Recall that $G$ is a $C^1$ operator and solutions of the BVP are equivalent to zeros of $G$.

Let $\tilde y$ be the approximate solution on the mesh $\{t_i\}_{i=0}^N$. We assume that for each $i$, the restriction  to the subinterval $[t_{i-1},t_i]$, denoted $y_i(t)$, is a polynomial,
\begin{align}
\tilde y_i(t)=\sum_{k=0}^m y_{i,k}(t-t_{i-1/2})^k.
\end{align}
In practice, we choose the $y_{i,k}$ so that the ODE is satisfied up to degree $m-1$, but the bounds we present below don't depend on the precise definition of the $y_{i,k}$'s.

Our first ingredient for the Newton Kantorovich theorem is a bound on $\|G(\tilde y)\|=\max\{\|G(\tilde y)_1\|, |G(\tilde y)_2|\}$, where
\begin{align}
\|G(\tilde y)_1\|=&\max_i \sup_{t\in[t_{i-1},t_i]} \bigg|  y_i(t_{i-1})-y_1(0)- \sum_{j=1}^{i-1}\int_{t_{j-1}}^{t_j} f(s,y_j(s))ds+\int_{t_{i-1}}^t y_i^\prime(s)- f(s,y_i(s))ds\bigg|,\notag\\
\|G(\tilde y)_2\|=&|g(y_1(0),y_N(1))|.
\end{align}

Assuming that we have an interval extension of $g$, a verified bound on $|g(y_1(0),y_N(1))|$ is simple to compute.  As for the first component, assuming that $f$ is $C^m$, we obtain the following by employing the Taylor expansion with remainder
\begin{align}
&f(t,y_i(t))=\tilde f_i(t)+R_i^m(t)(t-t_{i-1/2})^m,
\end{align}
where we define
\begin{align}
\tilde f_i(t)\equiv \sum_{l=0}^{m-1}\sum_{|\delta|\leq m-1-l}\frac{1}{l!\delta!}\partial_t^l\partial_y^\delta f(t_{i-1/2},y_{i,0})(t-t_{i-1/2})^l\left(\sum_{k=1}^m y_{i,k}(t-t_{i-1/2})^k\right)^\delta
\end{align}
and
\begin{align}
R_i^m(t)\equiv &\sum_{l=0}^m\sum_{|\delta|=m-l} \frac{m}{l!\delta!} \left(\sum_{k=1}^m y_{i,k}(t-t_{i-1/2})^{k-1}\right)^\delta\\
&\times \int_0^1 (1-s)^{m-1}\partial_t^l\partial_y^\delta f\left(t_{i-1/2}+s(t-t_{i-1/2}),y_{i,0}+s\sum_{k=1}^m y_{i,k}(t-t_{i-1/2})^k\right)ds.\notag
\end{align}
Here, $\delta$ denotes a multi-index of length $n$.

With this we can write
\begin{align}
&\|G(\tilde y)_1\|\\
=&\max_i \sup_{t\in[t_{i-1},t_i]} \bigg| \sum_{j=1}^{i-1}\left(y_{j+1}(t_{j})- y_j(t_j)\right)+\sum_{j=1}^{i-1}\int_{t_{j-1}}^{t_j}y_j^\prime(s)-\tilde f_j(s)ds+\int_{t_{i-1}}^ty_i^\prime(s)- \tilde f_i(s)ds\notag\\
&- \sum_{j=1}^{i-1}\int_{t_{j-1}}^{t_j}R_j^m(s)(s-t_{j-1/2})^mds-\int_{t_{i-1}}^t R_i^m(s)(s-t_{i-1/2})^mds \bigg|.\notag
\end{align}

The integrals on the second line involve polynomials and can be evaluated explicitly, making an interval enclosure of these terms simple to obtain  (assuming that we have interval extensions for $f$ and its derivatives). The integrals on the third line can be enclosed by using the lemmas in  \ref{app:int_bound}.

Next we need a bound on $K$, the Lipschitz constant for $DG$, over a neighborhood $\overline{B_\epsilon(\tilde y)}$.  As before, we will only concern ourselves with diagonal weight matrix, $W=\text{diag}(w)$.  In this case, we have the bound
 \begin{align}
K\leq&\max\bigg\{ \max_i  \max_\alpha \sup\sum_{k,l} |W^\alpha_\alpha  (\partial_{y^k}\partial_{y^l}f)^\alpha([t_{i-1},t_i], y_i( [t_{i-1},t_i])+ C^w_\epsilon)(W^{-1})^k_k  (W^{-1})^l_l|,\notag\\
&\sum_{\delta,\gamma=1}^2    \max_\alpha \sup \sum_{k,l} |W^\alpha_\alpha ( \partial_{y_\delta^k}\partial_{y_\gamma^l}g)^\alpha(y_1(0)+ C^w_\epsilon, y_N(1)+C^w_\epsilon)(W^{-1})^k_k  (W^{-1})^l_l|\bigg\},\notag\\
C^w_\epsilon\equiv& \prod_i [-\epsilon/w_i, \epsilon/w_i].
\end{align}
In the second line, $y_1^j$ refers to the components of the first argument of $g$ and $y_2^j$ to the second.

All that remains is to prove invertibility and bound the inverse of $DG(\tilde y)$ using Theorem \ref{main_theorem}.  To do this, we need to break $A(t)\equiv D_yf(t,\tilde y(t))$ into a polynomial part and remainder on each subinterval.  This can be done by using the Taylor series with remainder once more
\begin{align}
A_i(t)=&D_yf(t,y_i(t))\\
=&\sum_{l=0}^{M-1}\sum_{|\delta|\leq M-1-l}\frac{1}{l!\delta!}\partial_t^l\partial_y^\delta D_yf(t_{i-1/2},y_{i,0})(t-t_{i-1/2})^l\left(\sum_{k=1}^m y_{i,k}(t-t_{i-1/2})^k\right)^\delta\notag\\
&+(t-t_{i-1/2})^M\sum_{l=0}^M\sum_{|\delta|=M-l}\frac{m}{l!\delta!}\left(\sum_{k=1}^m y_{i,k}(t-t_{i-1/2})^{k-1}\right)^\delta\notag\\
&\times\int_0^1 (1-s)^{m-1}\partial_t^l\partial_y^\delta D_yf\left(t_{i-1/2}+s(t-t_{i-1/2}),y_{i,0}+s\sum_{k=1}^m y_{i,k} (t-t_{i-1/2})^k\right)ds.\notag
\end{align}
Again, $\delta$ denotes a multi-index of length $n$.  Note that we allow $M$ and $m$ to be unrelated, with the only constraint being sufficient smoothness of $f$.

The polynomial part of $A_i(t)$ is then obtained by computing the interval enclosures of the polynomial coefficients up to degree $M-1$ .  As for the higher order terms, after taking out a common factor of $(t-t_0)^{M}$, we compute a matrix-interval enclosure over $[t_{i-1},t_i]$ for inclusion in the remainder term. The integral terms are also incorporated into the remainder using  \ref{app:int_bound}
\begin{align}
&\int_0^1 (1-s)^{m-1}\partial_t^l\partial_y^\delta D_yf\left(t_{i-1/2}+s(t-t_{i-1/2}),y_{i,0}+s\sum_{k=1}^m y_{i,k} (t-t_{i-1/2})^k\right)ds\notag\\
&\times \left(\sum_{k=1}^m y_{i,k}(t-t_{i-1/2})^{k-1}\right)^\delta\notag\\
\in &\frac{1}{m}\partial_t^l\partial_y^\delta D_yf\left([t_{i-1},t_i],\sum_{k=0}^m y_k [-\Delta t_i,\Delta t_i]^k\right) \left(\sum_{k=1}^m y_{i,k}[-\Delta t_i,\Delta t_i]^{k-1}\right)^\delta,
\end{align}
where $\Delta t_i=t_i-t_{i-1/2}$.    

These are all the pieces needed to use Theorem \ref{Newton} to prove existence of a zero of $G$ close to $\tilde y$.

\section{Numerical Tests}\label{sec:num_tests}

In (\ref{I-FH_bound2_1}) and (\ref{H_bound}) bounds were given on $\|I-FH\|$ and $\|H\|$ that involve finitely representable objects and a finite number of arithmetic operations with which we can perform validated computations using interval arithmetic.  For rigorous bounds, we implemented these formulas using the MATLAB interval analysis package INTLAB \cite{Ru99a}; the code that produces the desired bound on $\|I-FH\|$ is a simple translation of (\ref{I-FH_bound2_1}) and (\ref{I-FH_bound2_2}).

In this section, the applicability and accuracy of the BVP method is assessed using a pair of singularly perturbed test problems, taken from \cite{Lee}, as well as a nonlinear example, the existence of a periodic orbit in the Lorenz system.

To obtain usable bounds, even the exponentially small parts of the approximate fundamental solution need to be resolved with the same relative accuracy as the larger portions, otherwise the loss of accuracy will prevent the growing modes from being annihilated when forming the approximate Green's function.  For problems with quickly decaying modes, the absolute tolerance of the non-rigorous solver must be set near the underflow limit in order to achieve this.  This is done in the first two test problems.

 Tests were performed  using a uniform mesh. There is likely room for improvement by intelligently selecting the mesh but we did not employ such methods here. In table \ref{table:tests} the mesh size was chosen by adding points until the error approximately stabilized. 

\subsection{Example 1: Turning Point}
Our first test problem is a singularly perturbed Airy equation
\begin{equation*}
\epsilon v^{\prime\prime}-(t-1/2) v=0, \hspace{2mm} v(0)=v(1)=1,
\end{equation*}
which is a model of the quantum mechanical wave function near a turning point in the potential, here at $t=1/2$. The exact solution is a linear combination of Airy functions and is shown in figure \ref{fig:airy_sol}, where $\epsilon=10^{-6}$.  The solution exhibits dense oscillations for $t<1/2$ and a boundary layer near $t=1$.   

We employed the BVP method in the cases $\epsilon=10^{-4},10^{-5},10^{-6}$. We used $m=10$ with the aim of making the Taylor series and approximate inverse errors negligible, allowing us to focus on how the error in the initial approximation translates into the bound produced by the method. 

We define the weight matrix $W$ so that it balances the different scales of the jump terms.  More specifically, our strategy for linear problems is to use a diagonal $W$ with entries chosen so that
\begin{equation*}
W^i_i\sum_{k=1}^{N-1}|(\tilde y_{k+1}(t_{k})-\tilde y_k(t_k)+\tilde r_k(t_k)-\tilde r_{k+1}(t_{k}))^i|=\text{constant}, \hspace{2mm}\max_i W_i^i=1,
\end{equation*}
where $|\cdot|$ denotes  the absolute value. For this and the next test problem, $y=(v,v^\prime)$ and $r=0$.  The above choice of weight prevents overestimation of the error in certain components when they have widely different scales and can significantly improve the computed bounds, up to several orders of magnitude in our tests; it was often the difference between the existence test $\|I-FH\|<1$ passing or failing. Since the jump errors can be computed in $\mathcal{O}(N)$ operations prior to the main $\mathcal{O}(N^2)$ computation, this is an effective and practical way to improve the estimates. 

The test results are shown  in table  \ref{table:tests}. The fourth column gives an estimate of the (unweighted) error for the first component of the solution at the center of the mesh intervals, $v_i$, that are input into the BVP method.  The approximate solution was generated by the matlab bvp5c routine with relative error tolerance of $10^{-6}$. The fifth column is a bound on $\|I-FH\|$ computed by the BVP method, and the final column is the computed  bound on the absolute value of the (unweighted) error of the first solution component, $v$.    

The a posteriori BVP method does overestimate the error by several orders of magnitude, but the error bounds are still practical. The IVP fundamental solution method would not be usable in this situation, as the fundamental solution matrix has a norm $|Y(1)|\approx 10^{12}$  for $\epsilon=10^{-4}$. As an illustration of the usefulness of the weight, without it, the error bound for $\epsilon=10^{-6}$ is weakened from $1.8\times 10^{-4}$ to $44$.

The BVP method fails on this problem for $\epsilon=10^{-7}$, as $\Phi$ becomes so ill conditioned that its decaying components cannot be resolved to sufficient relative accuracy due to underflow. This point will reappear in the subsequent example as well; the exponentially decaying modes of $\Phi$ (which tend to be the least `interesting' features) cause the most trouble for the BVP method. In other words, while a large norm for the inverse fundamental solution does not impact BVP stability, over/underflow issues mean that it is still the main limiting factor of the BVP method in its current form. However, the problem is much less severe than ill conditioning in the IVP method.

%%%%%%%%%%%%%%%%%%%%%%%%%%%%%%%%%%%%%%%
\begin{figure}
\centerline{\includegraphics[height=6cm]{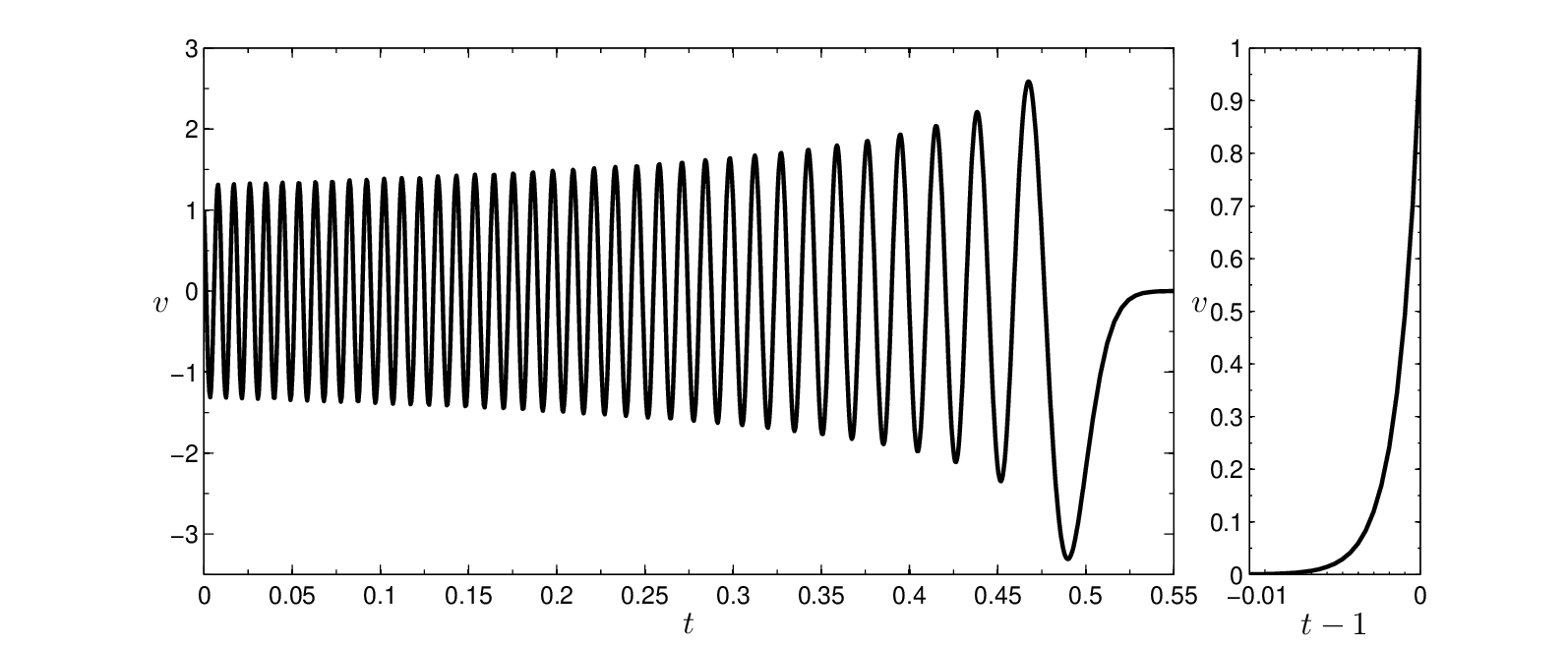}}
\caption{Exact solution to the turning point problem for $\epsilon=10^{-6}$.  For $t<1/2$ the system is oscillatory (left) and it has a boundary layer near $t=1$ (right).}\label{fig:airy_sol}
\end{figure}
%%%%%%%%%%%%%%%%%%%%%%%%%%%%%%%%%%%%
\begin{table}
\centering
\begin{tabular}{ |c |c |c |c|c|c|}
\hline
 \rule{0pt}{2.3ex}   Example & $\epsilon$ & N & input error &$\|I-FH\|$& solution error bound \\
\hline
\rule{0pt}{2.5ex}     Turning Point &$10^{-4}$& $150$&$9.8\times 10^{-10}$ &$ 1.1\times 10^{-6}$&$1.2\times 10^{-7}$\\
\hline
\rule{0pt}{2.5ex}      &$10^{-5}$& $250$  &$6.2\times 10^{-8}$ & $3.6\times 10^{-5}$ & $4.2\times 10^{-5}$\\
\hline
\rule{0pt}{2.5ex}      &$10^{-6}$ & $350$  & $1.3\times 10^{-6} $& $2.6\times 10^{-3}$&$1.8\times 10^{-4}$\\
\hline
\rule{0pt}{2.5ex}     Potential Barrier  & $10^{-5}$ & $150$ & $1.1\times 10^{-6}$ &$9.7\times 10^{-5}$ & $2.9\times 10^{-3}$\\
\hline
\rule{0pt}{2.5ex}        & $10^{-6}$ & $350$ & $ 3.0\times 10^{-8}$ &$8.1\times 10^{-6}$ & $1.6\times 10^{-6}$\\
\hline
\end{tabular}
\vspace{.3mm}
\caption{Test results for the BVP method on the example problems.}\label{table:tests}
\end{table}

\subsection{Example 2: Potential Well}
As a second test,  consider a quantum mechanical potential well problem
\begin{equation*}
\epsilon v^{\prime\prime}+((t-1/2)^2-\omega^2)v=0,\hspace{2mm} v(0)=1,\hspace{2mm} v(1)=2
\end{equation*}
where we take $\omega=1/4$.   The solution is oscillatory on $[0,1/4]\cup[3/4,1]$ and exponentially decaying as one moves towards the center of $[1/4,3/4]$ as seen in figure \ref{fig:potential} (right pane) for $\epsilon=10^{-6}$ and $m=10$. The definition of $W$ and the meaning of the table entries is the same as in the prior example. Again, the method fails once the decaying modes start to underflow, which occurs for $\epsilon=10^{-7}$.

%%%%%%%%%%%%%%%%%%%%%%%%%%%%%%%%%%%%%%%
\begin{figure} 
\centerline{\includegraphics[height=6cm]{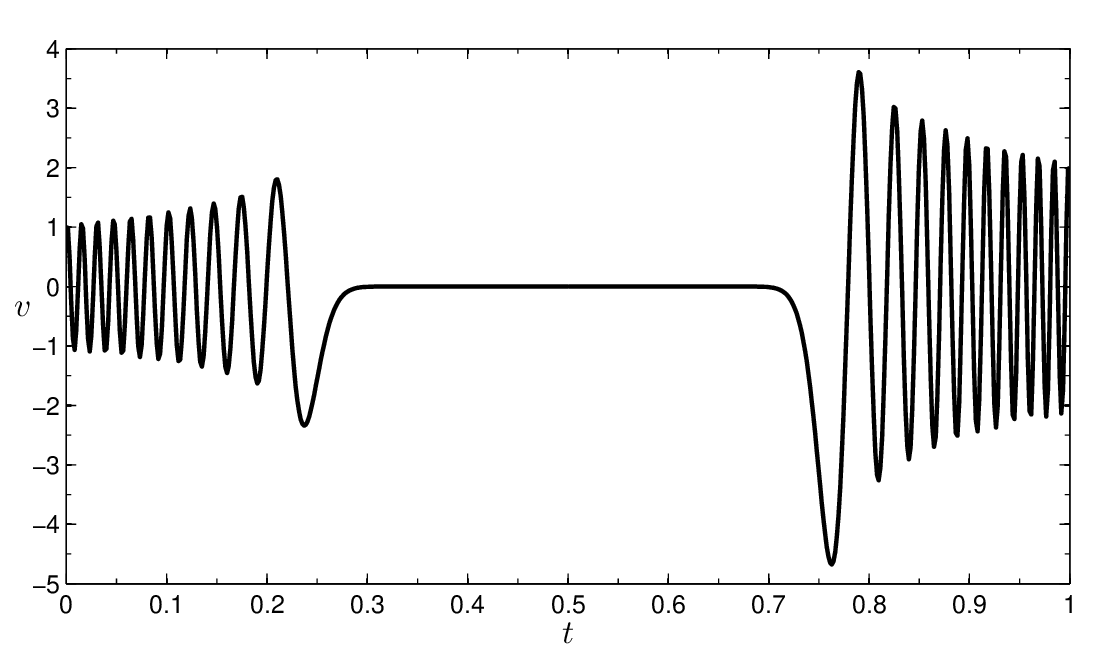}}
\caption{Solution to the potential problem for $\epsilon=10^{-6}$ (right).}\label{fig:potential}
 \end{figure}
%%%%%%%%%%%%%%%%%%%%%%%%%%%%%%%%%%%%%%%

 In the first  example, the supremum norm of the exact solution was relatively consistent between the different parameter values, but that is not true for the current example, where $\|v\|\approx 300$ for $\epsilon=10^{-5}$ and is single digits $\epsilon=10^{-6}$.  This helps explain why the absolute error is smaller for the smaller value of $\epsilon$.

\subsection{Example 3: Lorenz System}
As our final example, we study a non-linear problem, the existence of periodic solutions to the Lorenz system.  Scaling the independent variable by the (unknown) period $T$, the BVP is 
\begin{align*}
x^\prime=&T\sigma(y-x),\hspace{2mm} y^\prime=T(x(\rho-z)-y),\hspace{2mm} z^\prime=T(xy-\beta z),\hspace{2mm} T^\prime=0
\end{align*}
\begin{equation*}
x(0)=x(1),\hspace{2mm} y(0)=y(1),\hspace{2mm} z(0)=z(1).
\end{equation*}
In order to fix the initial point on the periodic orbit an additional boundary condition is needed, which we take to be
\begin{equation}\label{phase_cond}
x(0)=y(0), \text{ i.e. }x^\prime(0)=0.
\end{equation}
Let $f,g:\mathbb{R}^4\rightarrow\mathbb{R}^4$ denote the vector field and boundary conditions respectively and $G$ be the corresponding BVP operator, defined as in (\ref{G_def}).

In \cite{Coomes,radii_poly}, computer aided proof techniques were used to show existence of periodic orbits to the Lorenz system.  In particular, in~\cite{Coomes} existence was shown for parameter values
\begin{equation*}
\sigma=10,\hspace{2mm} \beta=8/3,\hspace{2mm} \rho=28
\end{equation*}
and initial conditions and period $T$ close to
\begin{equation*}
x(0)\approx-12.78619,\hspace{2mm}y(0)\approx-19.36419,\hspace{2mm}z(0)\approx24,\hspace{2mm} T\approx1.559.
\end{equation*}
Note that our initial point on the orbit, defined by (\ref{phase_cond}), differs from the above initial conditions. We will test our method on this same orbit, shown in the $x$-$y$ plane in figure~\ref{fig:Lorenz_orbit}.\\
%%%%%%%%%%%%%%%%%%%%%%%%%%%%%%%%%%%%%%%
\begin{figure}
\centerline{\includegraphics[height=6cm]{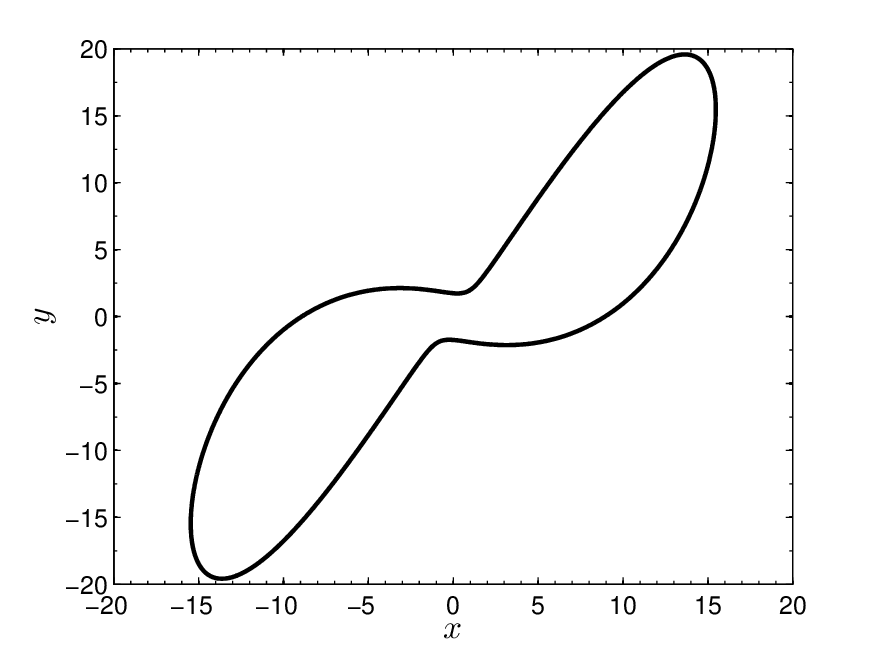}}
\caption{Periodic orbit for the Lorenz system, projected onto the $x$-$y$ plane.}\label{fig:Lorenz_orbit}
\end{figure}
%%%%%%%%%%%%%%%%%%%%%%%%%%%%%%%%%%%%

A Taylor series approximation to the solution centered at $t_{i-1/2}$ can be derived
\begin{equation*}
x_i(t)=\sum_{j=0}^m x_{i,j} (t-t_{i-1/2})^j,\hspace{2mm} y_i(t)=\sum_{j=0}^m y_{i,j} (t-t_{i-1/2})^j,\hspace{2mm} z_i(t)=\sum_{j=0}^m z_{i,j} (t-t_{i-1/2})^j,
\end{equation*}
\begin{align*}
x_{i,r+1}=&\frac{1}{r+1}T\sigma\left( y_{i,r} - x_{i,r}\right),\hspace{2mm}y_{i,r+1}=\frac{1}{r+1} T\left(\rho x_{i,r}-y_{i,r} -\sum_{j=0}^r x_{i,j} z_{i,r-j}\right),\hspace{2mm}\\
z_{i,r+1}=&\frac{1}{r+1}T\left(\sum_{j=0}^rx_{i,j} y_{i,r-j}-\beta z_{i,r}\right),
\end{align*}
where $x_{i,0}$, $y_{i,0}$, $z_{i,0}$, and $T$ are obtained from numerical approximations to the periodic orbit at the centers of the mesh intervals $[t_{i-1},t_i]$.

Using Section \ref{sec:nonlin}, we can compute upper bounds on the values of the parameters $\beta$, $K$, $\eta$ that appear in the Newton-Kantorovich theorem, see  \ref{newton_th}. 

  Denote the piecewise polynomial approximate solution by $\tilde Y$, let $F=DG(\tilde Y)$, and $H$ be its approximate inverse as in Theorem \ref{main_theorem}. Using $W=I$, $m=15$ and $M=5$, an equally spaced mesh of size $N=50$, and a domain $D=B_\epsilon(\tilde Y)$, $\epsilon\approx 5.1\times 10^{-7}$, our method provides a rigorous existence proof with the following parameters
\begin{align*}
\|G(\tilde Y)\|\leq& 4.0\times 10^{-10},\hspace{2mm} \|I-FH\|\leq 0.47,\hspace{2mm} \|F^{-1}\|\leq 6.6\times 10^2,\\
 K\leq  &79,\hspace{2mm} s_0\leq 2.6\times 10^{-7}, \hspace{2mm} s_1\geq 5.1\times 10^{-7}\notag.
\end{align*}
In other words, a solution exists within an $L^\infty$-norm ball of radius $s_0$ about the approximate solution and the solution is unique within the ball of radius $s_1$.

Again, the verified computation was done in MATLAB using the interval arithmetic package INTLAB \cite{Ru99a}. On a 2.2GHz Intel Core i5 the procedure took 150 seconds, including the time to compute the approximate solution $\tilde Y$ and approximate BVP fundamental solution $\tilde\Phi$.  The approximations were computed using the MATLAB routines ode15s and bvp5c.

\section{Conclusion}
We have presented an a posteriori method for proving existence to and deriving rigorous error bounds for two point boundary value problems on a compact interval, summarized in Theorem \ref{main_theorem}. The method applies to general $n$-dimensional systems without any assumption on the form of the vector field.  It uses a (non-rigorous) approximation to the Green's function of the BVP to generate the bounds and the bounds can be evaluated using interval arithmetic for mathematically rigorous results or in traditional floating point for faster but approximate results.  Because the method is based on the Green's function of a BVP rather than the fundamental solution to an IVP, the method is applicable to cases where the BVP is stable but the corresponding IVP is unstable. An adaptively chosen weight matrix was used in some of the norms and was found to improve the quality of the final $L^\infty$ error bounds by several orders of magnitude in many of the tests.

 In Section \ref{sec:num_tests} we tested the BVP method on a pair of singularly perturbed linear problems that exhibit such problematic features as dense oscillation and boundary layers.  The method is successful in proving existence and producing reasonable and usable error bounds, even when the relevant perturbation parameter becomes moderately small. We have also successfully applied it to rigorously prove the existence of a periodic orbit in the Lorenz system.  

In the cases tested, the main limiting factor on the success of the method is the extreme ill conditioning of the approximate BVP fundamental solution $\Phi$ as the parameter that controls the singular perturbation is made smaller.  Although $\Phi$ itself has moderate sized norm in the tests,  it becomes nearly singular.  The BVP method requires computation of the fundamental solution and its inverse to sufficient {\em relative accuracy}; poorly resolved decaying modes will spoil the method and the method always fails once the underflow limit is reached.  This commonly occurs for very stiff problems and so, while the BVP method works for moderately stiff problems as we have shown, it fails in the very stiff limit.   This issue could likely  be addressed in a brute force manner by employing an interval or floating point arithmetic package that allows for exponent values of substantially larger magnitude than are allowed in standard floating point arithmetic, but we did not employ such a technique here.

\appendix

\section{Newton-Kantorovich Theorem}\label{newton_th}
For convenience, in this appendix we state a version of the Newton-Kantorovich theorem that was  used above.
\begin{theorem}[Newton-Kantorovich Theorem]\label{Newton}
Let $X$, $Y$ be Banach spaces, $D\subset X$ be open and convex, $G:D\rightarrow Y$ be differentiable, and 
\begin{equation*}
\|DG(x)-DG(y)\|\leq K\|x-y\|
\end{equation*}
in $D$.  Let $\tilde y\in D$ and $DG(\tilde y)$ have bounded inverse $A$.

Suppose $\beta\geq \|A\|$, $\eta\geq \|AG(\tilde y)\|$, and $h\equiv \beta K \eta\leq 1/2$.  Set
\begin{equation*}
s_0=\frac{1}{\beta K}(1-\sqrt{1-2h}),\hspace{2mm} s_1=\frac{1}{\beta K}(1+\sqrt{1-2h}).
\end{equation*}
Suppose $S\equiv \overline{B_{s_0}(\tilde y)}\subset D$.  Then the Newton iteration is well defined, lies in $S$ and converges to $x\in S$, a solution of $G(x)=0$.  This is the unique solution in $D\cap B_{s_1}(\tilde y)$.  
\end{theorem}
In practice, we bound $\eta$ by bounding $\|A\|\|G(\tilde y)\|$.

\section{Computing $I-FH$}\label{app:I-FH}
This appendix outlines the derivation of a formula for $I-FH$.  The goal is decompose it into several terms, each of which obviously vanishes when $\tilde \Phi$ and $\tilde G$ are defined using the exact fundamental solution $\Phi$ and Green's function $G$, respectively.  The calculation is long but relatively straightforward.  We outline the key steps below.

 First, using the definitions of $H$ and $F$ and taking $t\not\in\{t_i\}_{i=0}^N$  the two components equal
\begin{align*}
&(I-FH)[r,w]_1(t)\\
=&\tilde\Phi(0)(w-B_1r(1))+\int_0^1 \tilde G(0,s)A(s)r(s)ds-\tilde\Phi(t)(w-B_1r(1))\\
&-\int_0^1 \tilde G(t,s)A(s)r(s)ds+\int_0^tA(s)\tilde\Phi(s)ds(w-B_1r(1))+\int_0^tA(s)r(s)ds\\
&+\int_0^tA(s)\int_0^1 \tilde G(s,z)A(z)r(z)dzds,\\
&(I-FH)[r,w]_2\\
=&[I-B_0\tilde\Phi(0)-B_1\tilde\Phi(1)](w-B_1r(1))-\int_0^1(B_0\tilde G(0,s) +B_1 \tilde G(1,s) )A(s) r(s) ds.
\end{align*}

One can check that 
\begin{align*}
&B_0 G(0,s)+B_1 G(1,s)=(-B_0\Phi(0)B_1\Phi(1)+B_1\Phi(1)B_0\Phi(0))\Phi^{-1}(s)\\
=&-(B_0\Phi(0)+B_1\Phi(1)-I)B_1\Phi(1)+B_1\Phi(1)(B_0\Phi(0)+B_1\Phi(1)-I).
\end{align*}
It is apparent that each term vanishes for the exact solution.

As for the first component, we introduce $\tilde\Phi^\prime$ and $\partial_t \tilde G(t,s)$ to obtain terms that vanish when the ODE is satisfied 
\begin{align*}
&(I-FH)[r,w]_1(t)\\
=&\tilde\Phi(0)(w-B_1r(1))+\int_0^1 \tilde G(0,s)A(s)r(s)ds-\tilde\Phi(t)(w-B_1r(1))-\int_0^1 \tilde G(t,s)A(s)r(s)ds\\
&+\int_0^t-\tilde\Phi^\prime(s)+A(s)\tilde\Phi(s)ds(w-B_1r(1))+\int_0^t\tilde\Phi^\prime(s)ds(w-B_1r(1))+\int_0^tA(s)r(s)ds\\
&+\int_0^t\int_0^1(-\partial_s \tilde G(s,z)+A(s) \tilde G(s,z))A(z)r(z)dzds+\int_0^t\int_0^1\partial_s \tilde G(s,z)A(z)r(z)dzds.
\end{align*}

Now integrate by parts, taking into account the discontinuities of $\tilde\Phi(s)$ and $\tilde G(s,t)$ at $s=t_i$ as well as $\tilde G(s,t)$ at $s=t$.  After substantial simplification, for $t\in[t_{i-1},t_i]$ we arrive at

\begin{align}\label{I_FH1_simp}
&(I-F H)[r,w]_{1}(t)\\
=&\sum_{j=1}^{i-1} \left(\tilde\Phi_j(t_j) - \tilde\Phi_{j+1}(t_{j}) \right)(w-B_1r(1))+\left(\int_{0}^{t} A(s) \tilde\Phi(s)-\tilde\Phi^\prime(s)ds\right)(w-B_1r(1))\notag\\
&+\int_{0}^{t} \int_0^1 (A(s)\tilde G(s,z)-\partial_s\tilde G(s,z)) A(z) r(z)dzds\notag\\
&+\sum_{j=1}^{i}\int_{t_{j-1}}^{\min\{t,t_j\}} (I-\tilde G_{j,j}^-(z,z)+\tilde G_{j,j}^+(z,z))A_j(z) r_j(z) dz\notag\\
&+\sum_{j=1}^{i-1} \sum_{k=1}^N\int_{t_{k-1}}^{t_k} ( \tilde G_{j,k}(t_j,z)- \tilde G_{j+1,k}(t_{j},z) )A_k(z) r_k(z)dz.\notag
\end{align}

The exact Greens function satisfies 
\begin{equation*}
G(t^+,t)-G(t^-,t)=\Phi(t)(B_1\Phi(1)+B_0\Phi(0))\Phi^{-1}(t)=I
\end{equation*}
and is continuous away from the diagonal.  The exact fundamental solution $\Phi$ is continuous as well.  Therefore, if $\tilde\Phi$ and $\tilde G$ are defined using the exact fundamental solution $\Phi$ and Green's function $G$ then the jump terms in (\ref{I_FH1_simp}) vanish.  The remaining terms are zero since $\Phi$ and $G$ satisfy the ODE.  Therefore we have succeeded in putting $I-FH$ into a form where each term is manifestly zero when using the exact solution.  This suggests that a reasonable bound on $\|I-FH\|$ might be obtained by bounding each of these terms individually, as done in Section~\ref{sec:sharper_bound}.

\section{Formula for $I-HF$ }\label{app:I-HF}
A similar procedure to  \ref{app:I-FH} yields a formula for $I-HF$.  For $t\in (t_{i-1},t_i)$
\begin{align*}
&(I-HF)[v](t)=\int_{0}^{1}\left(\partial_s \tilde G(t,s)+  \tilde G(t,s) A(s)\right) \left(v(0)+\int_0^s A(z)v(z)dz\right) ds\\
&+\left(I-\tilde G_{i,i}^-(t,t)+ \tilde G_{i,i}^+(t,t)\right) \left(v(0)+\int_0^{t} A(z)v(z)dz\right)\\
&+(\tilde G_{i,1}(t,0)-\tilde\Phi_i(t)B_0)v(0)-(\tilde G_{i,N}(t,1)+\tilde\Phi_i(t)B_1)\left(v(0)+\int_0^1A(s)v(s)ds\right)\\
&+\sum_{j=1}^{N-1}(\tilde G_{i,j+1}(t,t_j)-\tilde G_{i,j}(t,t_j))\left(v(0)+\int_0^{t_j} A(z)v(z)dz\right).
\end{align*}

 In practice we found that the bounds obtained using $I-FH$ to be tighter than those using $I-HF$ and so we have based the BVP method on $I-FH$, but the other choice could have been made and may be preferable in some cases.

\section{Approximate ODE Solution}\label{app:taylor}
As discussed above, the data for the BVP method consists of sets $\tilde\Phi_j$ and $\tilde G_{i,j}^{\pm}$ of approximations to the fundamental solution $\Phi$ and Green's function $G$ at the centers of the intervals $[t_{i-1},t_i]$ and rectangles $[t_{i-1},t_i]\times [t_{j-1},t_j]$ respectively.  Recall that $\tilde G_{i,j}^{+}$ refers to the upper triangular region and $\tilde G_{i,j}^{-}$ to the lower triangular region and that $\tilde G_{i,j}^{+}=\tilde G_{i,j}^{-}$ for $i\neq j$ and that we must have $\tilde G^+_{i,i}$ and $\tilde G^-_{i,i}$ for all $i$. 

The values at the nodes must be extended to $C^1$ functions on each interval or rectangle.  This is done by letting $P_i(t)$, $Q_i(t)$ be the matrix-valued polynomials of degree $m\geq 1$ that satisfies the ODEs
\begin{align*}
\frac{d}{dt}P_i(t)=&A_i(t)P_i(t), \hspace{2mm} P_i(t_{i-1/2})=I,\\
\frac{d}{dt}Q_i(t)=&-Q_i(t)A_i(t), \hspace{2mm} Q_i(t_{i-1/2})=I
\end{align*}
up to order $m-1$, where $t_{i-1/2}$ denotes the midpoint of $[t_{i-1},t_i]$.  The functions
\begin{equation*}
\tilde\Phi_i(t)=P_i(t)\tilde\Phi_i,\hspace{2mm} \tilde G^{\pm}_{i,j}(t,s)=P_i(t)\tilde G^{\pm}_{i,j}Q_j(s)
\end{equation*}
are then approximate solutions to the ODEs satisfied by $\Phi$ and $G$.  

  Recursive formulas for the coefficients are given below. We suppose the coefficient matrix $A_i(t)$ is $C^{m}$ (the subscript $i$ refers to the restriction to $[t_{i-1},t_i]$ and write the Taylor series with remainder about the midpoint $t_{i-1/2}$ as
\begin{equation*}
A_i(t)=\sum_{k=0}^{m-1} A_i^k (t-t_{i-1/2})^k+A_i^{m}(t)(t-t_{i-1/2})^{m}.
\end{equation*}
Define $P_i(t)$, $Q_i(t)$ on $[t_{i-1},t_i]$ by
\begin{align}
P_i(t)=\sum_{k=0}^m(t-t_{i-1/2})^kP^k_i,\hspace{2mm} P_i^0=I,\hspace{2mm} P_i^k=\frac{1}{k}\sum_{l=0}^{k-1}A_i^{l}P_i^{k-l-1},\label{Pj_def}\\
Q_i(t)=\sum_{k=0}^m(t-t_{i-1/2})^kQ_i^k,\hspace{2mm} Q_i^0=I,\hspace{2mm} Q_i^{k}=-\frac{1}{k}\sum_{l=0}^{k-1}Q_i^{k-l-1}A_i^{l}.\label{Qj_def}
\end{align}
  Note that these quantities do not involve the remainder term $A_i^{m}(t)$.

With these,  the ODE residuals are
\begin{align}
&A_i(t)P_i(t)-P_i^\prime(t)=(t-t_{i-1/2})^m\sum_{k=0}^{m} (t-t_{i-1/2})^k \sum_{l=0}^{m-k} A_i^{l+k} P^{m-l}_i\equiv R_i(t)(t-t_{i-1/2})^m,\label{P_ODE_resid}\\
&Q_i^\prime(t)+Q_i(t)A_i(t)=(t-t_{i-1/2})^m\sum_{k=0}^{m}(t-t_{i-1/2})^{k}  \sum_{l=0}^{m-k}Q_i^{m-l} A_i^{l+k}.
\end{align}

Our construction also ensures that $P_i(t)$ and $Q_i(t)$ are approximate inverses to one another
\begin{align}
P_i(t)Q_i(t)-I=(t-t_{i-1/2})^{m+1}\sum_{k=1}^{m}  (t-t_{i-1/2})^{k-1}  \sum_{l=0}^{m-k} P^{m-l}_iQ_i^{l+k}\equiv \tilde R_i(t)(t-t_{i-1/2})^{m+1}.\label{PQ_I_resid}
\end{align}

\section{Interval Enclosures of Integrals} \label{app:int_bound}

The following lemmas are useful for deriving verified bounds of integrals.  We provide a simple proof of the first to illustrate the idea.  The other is similar.
\begin{lemma}
Let $g:[a,b]\to [0,\infty)$ and $R:[a,b]\to \mathbb{R}^n$ both be integrable. In addition, suppose  we have an interval-vector enclosure, $R(t)\in[ x_1, x_2]$, where $x_i\in\mathbb{R}^n$.  Then 
\begin{align}
\int_a^b g(t)R(t) dt\in \int_a^b g(t)dt \cdot [{ x_1},{ x_2}].
\end{align}
\end{lemma}
\begin{proof}
The components of $R$ satisfy $x_1^i\leq R^i(t)\leq x_2^i$.  $g(t)$ is non-negative, therefore
\begin{align}
g(t)x_1^i\leq g(t) R^i(t)\leq g(t)x_2^i.
\end{align}
Hence
\begin{align}
\int_a^bg(t)dt \cdot x_1^i\leq \int_a^bg(t) R^i(t)dt\leq \int_a^bg(t)dt\cdot x_2^i,
\end{align}
which is equivalent to the claim.
\end{proof}

\begin{lemma}
Let $R:[t_0-\Delta t,t_0+\Delta t]\rightarrow\mathbb{R}^n$  have the interval enclosure $R(t)\in[ x_1, x_2]$ where $x_i$ are vectors and so the r.h.s. is an interval vector.  Then for any $m\geq 0$ and $t\in [t_0-\Delta t,t_0+\Delta t]$ we have
\begin{align}
\int_{t_0-\Delta t}^{ t} R(t) (t-t_0)^m dt= (-1)^m [x_1,x_2]\frac{[0,\Delta t]^{m+1}}{m+1}+[x_1 ,x_2]\frac{[0,\Delta t]^{m+1}}{m+1}.
\end{align}
\end{lemma}

\section*{Acknowledgments}
 This work was conducted with Government support under and awarded by DoD, Air Force Office of Scientific Research, National Defense Science and Engineering Graduate (NDSEG) Fellowship, 32 CFR 168a.

\providecommand{\href}[2]{#2}
\providecommand{\arxiv}[1]{\href{http://arxiv.org/abs/#1}{arXiv:#1}}
\providecommand{\url}[1]{\texttt{#1}}
\providecommand{\urlprefix}{URL }

\medskip
% The data information below will be filled by AIMS editorial staff
Received xxxx 20xx; revised xxxx 20xx.
\medskip

\end{document}